\documentclass[preprint]{imsart}

\RequirePackage{amsthm,amsmath,amsfonts,amssymb}
\RequirePackage[numbers]{natbib}
\RequirePackage[colorlinks,citecolor=blue,urlcolor=blue]{hyperref}
\RequirePackage{graphicx}

\usepackage{algorithm}
\usepackage{algpseudocode}

\theoremstyle{remark}

\newtheorem{thm}{Theorem}

\theoremstyle{definition}
\newtheorem{ex}[thm]{Example}

\providecommand{\abs}[1]{\lvert#1\rvert}

\newcommand\independent{\protect\mathpalette{\protect\independenT}{\perp}}
\def\independenT#1#2{\mathrel{\rlap{$#1#2$}\mkern2mu{#1#2}}}

\begin{document}

\begin{frontmatter}

\title{New perspectives on knockoffs construction}
\runtitle{Knockoffs}

\begin{aug}
\author[A]{\fnms{Patrizia} \snm{Berti}\ead[label=e1]{patrizia.berti@unimore.it}}
\author[B]{\fnms{Emanuela} \snm{Dreassi}\ead[label=e2]{emanuela.dreassi@unifi.it}}
\author[C]{\fnms{Fabrizio} \snm{Leisen}\ead[label=e3]{fabrizio.leisen@gmail.com}}
\author[D]{\fnms{Luca} \snm{Pratelli}\ead[label=e4]{pratel@mail.dm.unipi.it}}
\and
\author[E]{\fnms{Pietro} \snm{Rigo}\ead[label=e5]{pietro.rigo@unibo.it}}
\address[A]{Dipartimento di Matematica Pura ed Applicata ``G. Vitali'', Universit\`a di Modena e Reggio-Emilia, via Campi 213/B, 41100 Modena, Italy}

\address[B]{Dipartimento di Statistica, Informatica, Applicazioni, Universit\`a di Firenze, viale Morgagni 59, 50134 Firenze, Italy}

\address[C]{School of Mathematical Sciences, University of Nottingham, University Park, Nottingham, NG7 2RD, UK}

\address[D]{Accademia Navale, viale Italia 72, 57100 Livorno,
Italy}

\address[E]{Dipartimento di Scienze Statistiche ``P. Fortunati'', Universit\`a di Bologna, via delle Belle Arti 41, 40126 Bologna, Italy}
\end{aug}

\begin{abstract}
Let $\Lambda$ be the collection of all probability distributions for $(X,\widetilde{X})$, where $X$ is a fixed random vector and $\widetilde{X}$ ranges over all possible knockoff copies of $X$ (in the sense of \cite{CFJL18}). Three topics are developed in this paper: (i) A new characterization of $\Lambda$ is proved; (ii) A certain subclass of $\Lambda$, defined in terms of copulas, is introduced; (iii) The (meaningful) special case where the components of $X$ are conditionally independent is treated in depth. In real problems, after observing $X=x$, each of points (i)-(ii)-(iii) may be useful to generate a value $\widetilde{x}$ for $\widetilde{X}$ conditionally on $X=x$.
\end{abstract}



\end{frontmatter}

\section{Introduction}\label{intro}

The availability of massive data along with new scientific problems have reshaped statistical thinking and data analysis. High-dimensionality has significantly challenged the boundaries of traditional statistical theory, in particular in the regression framework. Variable selection methods are fundamental to discover meaningful relationships between an outcome and all the measured covariates.

A new approach to regression problems, hereafter referred to as the {\em knockoff procedure} (KP), has been recently introduced by Barber and Candes; see \cite{BC15}, \cite{BCS20}, \cite{BCJW21}, \cite{CFJL18}, \cite{SSC19}. KP aims to control the false discovery rate among all the variables included in the model. Indeed, KP is relevant for at least two reasons. Firstly, there are not many variable selection methods able to control the false discovery rate with finite-sample guarantees, mainly when the number $p$ of covariates far exceeds the sample size $n$. Secondly, KP makes assumptions that are substantially different from those commonly encountered in a regression set up.

Let $X_i$ and $Y$ be real random variables, where $i=1,\ldots,p$ for some integer $p\ge 2$. Here, the $X_i$ should be regarded as covariates and $Y$ as the response variable. Letting
\begin{gather*}
X=(X_1,\ldots,X_p),
\end{gather*}
one of the main features of KP is to model the probability distribution of $X$ rather than the conditional distribution of $Y$ given $X$. Quoting from \cite[p. 554]{CFJL18}:

\vspace{0.2cm}

{\em The usual set-up for inference in conditional models is to assume a strong parametric model for
the response conditional on the covariates, such as a homoscedastic linear model, but to assume
as little as possible about, or even to condition on, the covariates. We do the exact opposite
by assuming that we know everything about the covariate distribution but nothing about the
conditional distribution $Y|X_1,\ldots,X_p$. Hence, we merely shift the burden of knowledge. Our
philosophy is, therefore, to model $X$, not $Y$, whereas, classically, $Y$ (given $X$) is modelled and
$X$ is not.}

\vspace{0.2cm}

Real situations where to model $X$ is more appropriate than to model $Y|X$ are actually common. An effective example, in a genetic framework, is in \cite[Sect. 1]{BCS20}.

As highlighted, the main target of KP is variable selection, taking the false discovery rate under control. We refer to \cite{BC15}, \cite{BCS20}, \cite{BCJW21}, \cite{CFJL18}, \cite{SSC19} for a description of KP and a discussion of its statistical behavior. In this paper, we deal with knockoff random variables, possibly the basic ingredient of KP.

\subsection{Two related problems} From now on, the probability distribution of any random element $U$ is denoted by $\mathcal{L}(U)$ and the coordinates of a point $x\in\mathbb{R}^n$ are indicated by $x_1,\ldots,x_n$. Moreover, we let
\begin{gather*}
I=\bigl\{1,\ldots,p\bigr\}.
\end{gather*}

For $x\in\mathbb{R}^{2p}$ and $S\subset I$, define $f_S(x)\in\mathbb{R}^{2p}$ by swapping $x_i$ with $x_{p+i}$ for each $i\in S$ and leaving all other coordinates fixed. Then, $f_S:\mathbb{R}^{2p}\rightarrow\mathbb{R}^{2p}$ is a permutation. For instance, for $p=2$, one obtains $f_S(x)=(x_3,x_2,x_1,x_4)$ if $S=\{1\}$, $f_S(x)=(x_1,x_4,x_3,x_2)$ if $S=\{2\}$ and $f_S(x)=(x_3,x_4,x_1,x_2)$ if $S=\{1,2\}$. Let
\begin{gather*}
\mathcal{F}=\bigl\{f_S:\,S\subset I\bigr\}
\end{gather*}
where $f_\emptyset$ is the identity map on $\mathbb{R}^{2p}$.

A {\em knockoff copy of} $X$, or merely a {\em knockoff}, is a $p$-variate random variable $\widetilde{X}=(\widetilde{X}_1,\ldots,\widetilde{X}_p)$ such that

\vspace{0.2cm}

\begin{itemize}

\item $f(X,\widetilde{X})\sim (X,\widetilde{X})$ for every $f\in\mathcal{F}$;

\vspace{0.2cm}

\item $\widetilde{X}\independent Y\mid X$.

\end{itemize}

\vspace{0.2cm}

The condition $\widetilde{X}\independent Y\mid X$ is automatically true if $\widetilde{X}=g(X)$ for some measurable function $g$. More generally, such a condition is guaranteed whenever $\widetilde{X}$ is constructed ``without looking" at $Y$. This is exactly the case of this paper. Hence, the condition $\widetilde{X}\independent Y\mid X$ is neglected.

We also note that a knockoff trivially exists. It suffices to let $\widetilde{X}=X$. This trivial knockoff, however, is not useful in practice. Roughly speaking, for KP to work nicely, $\widetilde{X}$ should be ``as independent of $X$ as possible".

Let $\Lambda$ denote the collection of all knockoff distributions, namely
\begin{gather*}
\Lambda=\bigl\{\mathcal{L}(X,\widetilde{X}):\widetilde{X}\text{ a knockoff copy of }X\bigr\}.
\end{gather*}

For KP to apply, a knockoff copy $\widetilde{X}$ of $X$ is required. Accordingly, the following two problems arise:

\vspace{0.2cm}

\begin{itemize}

\item[(i)] How to build a knockoff $\widetilde{X}$ ?

\vspace{0.2cm}

\item[(ii)] Is it possible to characterize $\Lambda$ ?

\end{itemize}

\vspace{0.2cm}

Questions (i) and (ii) are connected. A characterization of $\Lambda$, if effective, should suggest how to obtain $\widetilde{X}$. Anyhow, both (i) and (ii) have been answered.

As to (i), a first construction of $\widetilde{X}$ is Algorithm 1 of \cite[p. 563]{CFJL18}. Even if nice, however, this construction is not computationally efficient except from some special cases. See \cite[Sect. 2.2]{BCJW21}, \cite[Sect. 7.2.1 ]{CFJL18}, \cite[p. 6]{SSC19}.

As to (ii), a characterization of $\Lambda$ is in \cite[Theo. 1]{BCJW21}. Such a characterization, based on conditional distributions, is effective. In fact, exploiting it and the Metropolis algorithm, some further (efficient) constructions of $\widetilde{X}$ have been singled out.

\subsection{Our contribution}

This paper is about problems (i)-(ii). Three distinct issues are discussed.

\vspace{0.2cm}

\begin{itemize}

\item In Section \ref{k9g5}, a new characterization of $\Lambda$ is proved. Such a characterization is based on invariance arguments and provides a simple description of $\Lambda$. However, the characterization seems to have a theoretical content only. Apart from a few special cases, in fact, it does not help to build a knockoff in practice.

\vspace{0.2cm}

\item In Section \ref{c5g8n}, a certain (proper) subclass $\Lambda_0\subset\Lambda$ is introduced. The elements of $\Lambda_0$ admit a simple and explicit representation in terms of copulas. In particular, to work with $\Lambda_0$ is straightforward when $\mathcal{L}(X)$ corresponds to an Archimedean copula. Furthermore, if $\mathcal{L}(X,\widetilde{X})\in\Lambda_0$, the conditional distribution $\mathcal{L}(\widetilde{X}\mid X)$ can be written in closed form. Therefore, after observing $X=x$, a value $\widetilde{x}$ for $\widetilde{X}$ can be drawn from $\mathcal{L}(\widetilde{X}\mid X=x)$ directly. This is quite different from the usual methods for obtaining $\widetilde{x}$; see e.g. \cite{BC15}, \cite{BCS20}, \cite{BCJW21}, \cite{CFJL18}, \cite{SSC19}.

\vspace{0.2cm}

\item In Section \ref{2wx8m}, we focus on the case where $X_1,\ldots,X_p$ are conditionally independent, in the sense that
\begin{gather}\label{n8y6dd}
P\bigl(X_1\in A_1,\ldots,X_p\in A_p\bigr)=E\Bigl\{\prod_{i=1}^pP(X_i\in A_i\mid Z)\Bigr\}
\end{gather}
for some random element $Z$ and all Borel sets $A_1,\ldots,A_p\subset\mathbb{R}$. This section includes our main results. Indeed, under \eqref{n8y6dd}, to build a reasonable knockoff $\widetilde{X}$ is straightforward. In addition, with a suitable choice of $Z$, to realize condition \eqref{n8y6dd} is quite simple in practice. It suffices to regard $Z$ as a random parameter, equipped with a prior distribution, and to implement a sort of Bayesian procedure. For instance, to obtain a knockoff $\widetilde{X}$ such that cov$(X_i,\widetilde{X}_i)=0$ for each $i\in I$ is very easy; see Examples \ref{m2mx}-\ref{s6st8} for details. From the statistician's point of view, the advantage is twofold. Firstly, condition \eqref{n8y6dd} is easy to be realized and able to describe a number of real situations. Secondly, if $X$ is modeled by \eqref{n8y6dd}, to build $\widetilde{X}$ is straightforward. In particular, as in Section \ref{c5g8n}, the conditional distribution $\mathcal{L}(\widetilde{X}\mid X)$ can be usually written in closed form.
\end{itemize}

\vspace{0.2cm}

To close the paper, in Section \ref{e5ttf67n}, the results mentioned above are translated into practical algorithms. This section, written with applications in mind, aims to show how such results can be exploited in real problems.

\subsection{Further notation}\label{g7n9k}
 In the sequel,
\begin{gather*}
\widetilde{X}=(\widetilde{X}_1,\ldots,\widetilde{X}_p)
\end{gather*}
is {\em any} $p$-variate random variable (defined on the same probability space as $X$). Moreover, $\mathcal{B}_n$ is the Borel $\sigma$-field on $\mathbb{R}^n$ and $m_n$ the Lebesgue measure on $\mathcal{B}_n$.

\medskip

As in \cite[Sect. 4]{BPCID}, we denote by $\mathcal{S}(a,b)$ the {\em symmetric $\alpha$-stable law with parameters $a$ and} $b$, where $a\in\mathbb{R}$, $b>0$ and $\alpha\in (0,2]$. This means that $\mathcal{S}(a,b)$ is the probability distribution of $a+b^{1/\alpha}L$ where $L$ is a real random variable with characteristic function
\begin{gather*}
	E\bigl\{\exp(i\,t\,L)\bigr\}=\exp\Bigl(-\frac{\abs{t}^\alpha}{2}\Bigr)\quad\quad\text{for all }t\in\mathbb{R}.
\end{gather*}
Note that $\mathcal{S}(a,b)=\mathcal{N}(a,b)$ if $\alpha=2$ and $\mathcal{S}(a,b)=\mathcal{C}(a,b)$ if $\alpha=1$, where $\mathcal{C}(a,b)$ is the Cauchy distribution with density $f(x)=\frac{2\,b}{\pi}\,\frac{1}{b^2+4\,(x-a)^2}$ (the standard Cauchy distribution corresponds to $a=0$ and $b=2$).

\medskip

Finally, for any measures $\mu$ and $\nu$ (defined on the same $\sigma$-field) we write $\mu\ll\nu$ to mean that $\mu$ is absolutely continuous with respect to $\nu$, that is, $\mu(A)=0$ whenever $A$ is measurable and $\nu(A)=0$.

\section{A characterization of $\Lambda$}\label{k9g5}

Let $\mathcal{P}$ be the collection of $\mathcal{F}$-invariant probabilities, namely, those probability measures $\lambda$ on $\mathcal{B}_{2p}$ satisfying
\begin{gather*}
\lambda\circ f^{-1}=\lambda\quad\quad\text{for all }f\in\mathcal{F}.
\end{gather*}

We begin by noting that $\lambda\in\mathcal{P}$ if and only if
\begin{gather}\label{b7u}
\lambda=\frac{\sum_{f\in\mathcal{F}}\,\pi\circ f^{-1}}{2^p}
\end{gather}
for {\em some} probability measure $\pi$ on $\mathcal{B}_{2p}$. In fact, if $\lambda\in\mathcal{P}$, condition \eqref{b7u} trivially holds with $\pi=\lambda$ (since card$\,(\mathcal{F})=2^p$). Conversely, if $g\in\mathcal{F}$ and $\lambda$ meets \eqref{b7u} for some $\pi$, then
\begin{gather*}
\lambda\circ g^{-1}=\frac{\sum_{f\in\mathcal{F}}\,\pi\circ f^{-1}\circ g^{-1}}{2^p}=\frac{\sum_{f\in\mathcal{F}}\,\pi\circ (g\circ f)^{-1}}{2^p}=\lambda
\end{gather*}
where the last equality is because $\mathcal{F}$ is a group under composition.

The above characterization of $\mathcal{P}$ has the following consequence.

\begin{thm}\label{n7x}
$\lambda\in\Lambda$ if and only if condition \eqref{b7u} holds for some probability measure $\pi$ on $\mathcal{B}_{2p}$ such that
\begin{gather}\label{uhb}
\frac{1}{2^p}\,\sum_{f\in\mathcal{F}}\pi\big\{x\in\mathbb{R}^{2p}:f(x)\in A\times\mathbb{R}^p\bigr\}=P\bigl(X\in A)
\end{gather}
for each $A\in\mathcal{B}_p$.
\end{thm}

\begin{proof}
If $\lambda\in\Lambda$, conditions \eqref{b7u}-\eqref{uhb} trivially hold with $\pi=\lambda$. Conversely, under \eqref{b7u}-\eqref{uhb}, one obtains
\begin{gather*}
\lambda(A\times\mathbb{R}^p)=\frac{\sum_{f\in\mathcal{F}}\,\pi\circ f^{-1}(A\times\mathbb{R}^p)}{2^p}=P\bigl(X\in A)\quad\quad\text{for all }A\in\mathcal{B}_p.
\end{gather*}
Hence, up to enlarging the probability space where $X$ is defined, there exists a $p$-variate random variable $\widetilde{X}$ such that $\mathcal{L}(X,\widetilde{X})=\lambda$. Since $\lambda\in\mathcal{P}$ (because of condition \eqref{b7u}) $\widetilde{X}$ is a knockoff copy of $X$, namely, $\lambda\in\Lambda$.
\end{proof}

From the theoretical point view, Theorem \ref{n7x} provides a simple and clear description of $\Lambda$. Unfortunately, however, to select a probability $\pi$ satisfying condition \eqref{uhb} is very hard. Thus, in most cases, Theorem \ref{n7x} is not practically useful. Nevertheless, it may give some indications.

\begin{ex} {\bf ($X$ has a density with respect to a product measure).}\label{z2w9k}
Let
\begin{gather*}
\nu=\nu_1\times\ldots\times\nu_p
\end{gather*}
be a product measure on $\mathcal{B}_p$, where each $\nu_i$ is a $\sigma$-finite measure on $\mathcal{B}_1$. For instance, $\nu_i=m_1$ for all $i\in I$. Or else, $\nu_i=m_1$ for some $i$ and $\nu_j=\,$counting measure (on a countable subset of $\mathbb{R}$) for some $j$. And so on. In this example, we assume $\mathcal{L}(X)\ll\nu$. Hence, $X$ has a density $h$ with respect to $\nu$, namely
\begin{gather*}
P(X\in A)=\int_A h\,d\nu\quad\quad\text{for all }A\in\mathcal{B}_p.
\end{gather*}

Fix a probability measure $\pi$ on $\mathcal{B}_{2p}$ satisfying condition \eqref{uhb} and define $\lambda$ through condition \eqref{b7u}. Then, Theorem \ref{n7x} implies $\lambda\in\Lambda$. In addition, since the measure $\nu\times\nu$ is $\mathcal{F}$-invariant, one obtains $\lambda\ll\nu\times\nu$ provided $\pi\ll\nu\times\nu$. Precisely, if $\pi\ll\nu\times\nu$ and $g$ is a density of $\pi$ with respect to $\nu\times\nu$, then
\begin{gather*}
q=\frac{\sum_{f\in\mathcal{F}}\,g\circ f}{2^p}
\end{gather*}
is a density of $\lambda$ with respect to $\nu\times\nu$. This formula is practically useful. In fact, since $\lambda\in\Lambda$, there is a knockoff $\widetilde{X}$ such that $\mathcal{L}(X,\widetilde{X})=\lambda$. Hence, after observing $X=x$, a value $\widetilde{x}$ for such $\widetilde{X}$ can be drawn from the conditional density
\begin{gather*}
\frac{q(x,\widetilde{x})}{h(x)}=\frac{\sum_{f\in\mathcal{F}}g\bigl[f(x,\widetilde{x})\bigr]}{2^ph(x)}\quad\quad\text{where }x,\,\widetilde{x}\in\mathbb{R}^p.
\end{gather*}

Obviously, to make this example concrete, one needs a probability measure $\pi$ satisfying condition \eqref{uhb} and $\pi\ll\nu\times\nu$. As noted above, to find $\pi$ is usually hard. However, a probability $\pi$ with the required properties is in Example \ref{e6n8h}.
\end{ex}

\medskip

We close this section by determining those $\pi$ which satisfy equation \eqref{b7u} for a given $\lambda\in\mathcal{P}$.

\begin{thm}\label{r5y7bq}
Fix $\lambda\in\mathcal{P}$ and any probability measure $\pi$ on $\mathcal{B}_{2p}$. The following statements are equivalent:

\vspace{0.2cm}

\begin{itemize}

\item[(a)] Condition \eqref{b7u} holds, namely, $\lambda=\frac{\sum_{f\in\mathcal{F}}\,\pi\circ f^{-1}}{2^p}$;

\vspace{0.2cm}

\item[(b)] $\pi$ admits a density with respect to $\lambda$, say $q$, and
\begin{gather*}
\sum_{f\in\mathcal{F}}\,q\bigl[f(x)\bigr]=2^p\quad\quad\text{for }\lambda\text{-almost all }x\in\mathbb{R}^{2p};
\end{gather*}

\vspace{0.2cm}

\item[(c)] $\pi=\lambda$ on $\mathcal{G}$, where $\mathcal{G}=\bigl\{A\in\mathcal{B}_{2p}:f^{-1}(A)=A\text{ for all }f\in\mathcal{F}\bigr\}$.

\end{itemize}
\end{thm}

\vspace{0.2cm}

The proof of Theorem \ref{r5y7bq} is postponed to the final Appendix.

As an application of Theorem \ref{r5y7bq}, in the next example, $\lambda$ is a well known knockoff distribution and we look for a probability $\pi$ satisfying equation \eqref{b7u} with respect to $\lambda$.

\begin{ex}
Suppose $X\sim\mathcal{N}(0,\Sigma)$ and take a diagonal matrix $D$ such that
\begin{gather*}
G=\left(
    \begin{array}{cc}
      \Sigma & \Sigma-D \\
      \Sigma-D & \Sigma \\
    \end{array}
  \right)
\end{gather*}
is semidefinite positive. If $(X,\widetilde{X})\sim\mathcal{N}(0,G)$, then $\widetilde{X}$ is a knockoff copy of $X$; see e.g. \cite[p. 559]{CFJL18}. Fix $G$ as above and define $\lambda=\mathcal{N}(0,G)$. Define also
\begin{gather*}
q(x)=2^p\,\frac{\phi(x)}{\sum_{g\in\mathcal{F}}\phi[g(x)]}\quad\quad\text{for all }x\in\mathbb{R}^{2p},
\end{gather*}
where $\phi$ is any strictly positive Borel function on $\mathbb{R}^{2p}$. Since $\mathcal{F}$ is a group,
\begin{gather*}
2^{-p}\,\sum_{f\in\mathcal{F}}\,q\bigl[f(x)\bigr]=\sum_{f\in\mathcal{F}}\,\,\frac{\phi[f(x)]}{\sum_{g\in\mathcal{F}}\phi[g\circ f(x)]}=\frac{\sum_{f\in\mathcal{F}}\phi[f(x)]}{\sum_{g\in\mathcal{F}}\phi[g(x)]}=1.
\end{gather*}
Since card$\,(\mathcal{F})=2^p$ and $\lambda\in\mathcal{P}$,
\begin{gather*}
\int q(x)\,\lambda(dx)=\sum_{f\in\mathcal{F}}\,\int\,\frac{\phi(x)}{\sum_{g\in\mathcal{F}}\phi[g(x)]}\,\lambda(dx)
\\=\sum_{f\in\mathcal{F}}\,\int\,\frac{\phi[f(x)]}{\sum_{g\in\mathcal{F}}\phi[g(x)]}\,\lambda(dx)
=\int\,\frac{\sum_{f\in\mathcal{F}}\phi[f(x)]}{\sum_{g\in\mathcal{F}}\phi[g(x)]}\,\lambda(dx)=1.
\end{gather*}
Therefore, thanks to Theorem \ref{r5y7bq},
\begin{gather*}
\pi(dx)=q(x)\,\lambda(dx)
\end{gather*}
is a probability measure on $\mathcal{B}_{2p}$ satisfying equation \eqref{b7u}.
\end{ex}

\section{Constructing knockoffs via copulas}\label{c5g8n}

In this section, $F$ and $F_i$ are the distribution functions of $X$ and $X_i$, respectively. Moreover, for any distribution function $G$ on $\mathbb{R}^n$, we write $\lambda_G$ to denote the probability measure on $\mathcal{B}_n$ induced by $G$.

A $n$-{\em copula}, or merely a copula, is a distribution function on $\mathbb{R}^n$ with uniform (on the interval $(0,1)$) univariate marginals. By Sklar's theorem, for any distribution function $G$ on $\mathbb{R}^n$ there is a $n$-copula $C$ such that
\begin{gather*}
G(x)=C\bigl[G_1(x_1),\ldots,G_n(x_n)\bigr]\quad\quad\text{for all }x\in\mathbb{R}^n,
\end{gather*}
where $G_1,\ldots,G_n$ are the univariate marginals of $G$.

Let us {\em fix} a $p$-copula $C$ such that
\begin{gather*}
F(x)=C\bigl[F_1(x_1),\ldots,F_p(x_p)\bigr]\quad\quad\text{for all }x\in\mathbb{R}^p.
\end{gather*}
Note that $C$ is unique whenever $F_1,\ldots,F_p$ are continuous. Note also that, since $F$ is known, $C$ can be regarded to be known as well.

In order to manufacture a knockoff, a naive idea is to let
\begin{gather}\label{cmv}
H(x)=C\Bigl[D_1\bigl(F_1(x_1),F_1(x_{p+1})\bigr),\ldots,D_p\bigl(F_p(x_p),F_p(x_{2p})\bigr)\Bigr]
\end{gather}
for all $x\in\mathbb{R}^{2p}$, where $D_1,\ldots,D_p$ are any 2-copulas. Such an $H$ is a possible candidate to be the distribution function of $(X,\widetilde{X})$ for some knockoff copy $\widetilde{X}$ of $X$.

Unfortunately, $H$ may fail to be a distribution function on $\mathbb{R}^{2p}$. However $\lambda_H\in\Lambda$, more or less by definition, whenever $H$ is a distribution function and $D_1,\ldots,D_p$ are symmetric (i.e., $D_i(u_2,u_1)=D_i(u_1,u_2)$ for all $u\in[0,1]^2$ and $i\in I$).

\begin{thm}\label{c4b7}
Suppose that $H$ is a distribution function on $\mathbb{R}^{2p}$. Then,
\begin{gather}\label{de4q1l}
\lambda_H\big\{x\in\mathbb{R}^{2p}:f(x)\in A\times\mathbb{R}^p\bigr\}=P\bigl(X\in A)
\end{gather}
for all $f\in\mathcal{F}$ and $A\in\mathcal{B}_p$. In particular, $\lambda_H$ satisfies condition \eqref{uhb} (namely, \eqref{uhb} holds if $\pi=\lambda_H$). Moreover, $\lambda_H\in\Lambda$ whenever $D_1,\ldots,D_p$ are symmetric.
\end{thm}

\begin{proof}
To prove condition \eqref{de4q1l}, just note that
\begin{gather*}
\lim_{x_{p+1},\ldots,x_{2p}\rightarrow\infty} H\bigl[f(x)\bigr]=F(x_1,\ldots,x_p)\quad\quad\text{for all }f\in\mathcal{F}\text{ and }x\in\mathbb{R}^{2p}.
\end{gather*}
Moreover, if $D_1,\ldots,D_p$ are symmetric, then
\begin{gather*}
H\bigl[f(x)\bigr]=H(x)\quad\quad\text{for all }f\in\mathcal{F}\text{ and }x\in\mathbb{R}^{2p}.
\end{gather*}
Hence, $\lambda_H\in\mathcal{P}$ and Theorem \ref{n7x} implies $\lambda_H\in\Lambda$.
\end{proof}

For Theorem \ref{c4b7} to work, the obvious drawback is how to choose $D_1,\ldots,D_p$ in such a way that $H$ is a distribution function. However, when this drawback can be overcome, an explicit expression for $\mathcal{L}(X,\widetilde{X})$ is available where $\widetilde{X}$ is a knockoff copy of $X$. Hence, the conditional distribution of $\widetilde{X}$ given $X$ can be written in closed form.

\medskip

As an example, suppose that $H$ is a distribution function and $C,D_1,\ldots,D_p,F_1,\ldots,F_p$ are all absolutely continuous with respect to the Lebesgue measure of appropriate dimension. Then, $H$ is absolutely continuous with respect to the Lebesgue measure of dimension $2p$. Moreover, if $\widetilde{X}$ is such that $\mathcal{L}(X,\widetilde{X})=\lambda_H$, the conditional density of $\widetilde{X}$ given $X=x$ can be written as
\begin{gather*}
p(\widetilde{x}\mid x)=\frac{1}{\varphi\bigl[F_1(x_1),\ldots,F_p(x_p)\bigr]\,\prod_{i=1}^pf_i(x_i)}\,\cdot\,\frac{\partial^{2p} H}{\partial x_p\ldots\partial x_1\partial \widetilde{x}_p\ldots\partial \widetilde{x}_1}(x,\widetilde{x})
\end{gather*}
where $x,\,\widetilde{x}\in\mathbb{R}^p$ and $\varphi$ and $f_i$ are the densities of $C$ and $F_i$, respectively. This formula will be used in Section \ref{e5ttf67n}.

\medskip

We next discuss the choice of $D_1,\ldots,D_p$. As already noted, not every choice is admissible.

\begin{ex}\label{f7jn9} {\bf ($H$ may fail to be a distribution function).}
Let $p=2$ and $C(u)=\bigl(u_1+u_2-1)^+$ for $u\in [0,1]^2$. Then, with $D_1=C$, one obtains
\begin{gather*}
\lim_{x_4\rightarrow\infty}H(x)=C\Bigl[D_1\bigl(F_1(x_1),F_1(x_3)\bigr),\,D_2\bigl(F_2(x_2),1\bigr)\Bigr]
\\=C\Bigl[D_1\bigl(F_1(x_1),F_1(x_3)\bigr),\,F_2(x_2)\Bigr]
\\=\Bigl(F_2(x_2)+\bigl(F_1(x_1)+F_1(x_3)-1\bigr)^+-1\Bigr)^+\\=\bigl(F_1(x_1)+F_1(x_3)+F_2(x_2)-2\bigr)^+.
\end{gather*}
Therefore, $\lim_{x_4\rightarrow\infty}H(x)$ is not a distribution function on $\mathbb{R}^3$, so that $H$ is not a distribution function on $\mathbb{R}^4$.
\end{ex}

Let
\begin{gather*}
\Lambda_0=\bigl\{\lambda\in\Lambda:\text{ the distribution function of }\lambda\text{ admits representation \eqref{cmv}}\bigr\}.
\end{gather*}
Despite Example \ref{f7jn9}, a possible question is whether $\Lambda_0=\Lambda$.

\begin{ex} {\bf ($\Lambda_0$ is a proper subset of $\Lambda$).}
Let $U=(U_1,\ldots,U_{2p})$ and $V=(V_1,\ldots,V_{2p})$ be any random variables. Then, $\mathcal{L}(U)=\mathcal{L}(V)$ provided:
\begin{gather*}
\mathcal{L}(U)\in\Lambda_0,\,\mathcal{L}(V)\in\Lambda_0,\quad\text{and}\quad\mathcal{L}(U_i,U_{p+i})=\mathcal{L}(V_i,V_{p+i})\text{ for each }i\in I.
\end{gather*}
After noting this fact, take $U$ and $V$ exchangeable and such that
\begin{gather*}
\mathcal{L}(U)\ne\mathcal{L}(V)\quad\text{but}\quad\mathcal{L}(U_1,\dots,U_p)=\mathcal{L}(V_1,\dots,V_p).
\end{gather*}
Suppose also that $X\sim (U_1,\dots,U_p)$. Since $U$ is exchangeable, $\mathcal{L}(U)\in\mathcal{P}$. By Theorem \ref{n7x} and $X\sim (U_1,\dots,U_p)$, one obtains $\mathcal{L}(U)\in\Lambda$. Similarly, $\mathcal{L}(V)\in\Lambda$. Hence, at least one between $\mathcal{L}(U)$ and $\mathcal{L}(V)$ belongs to $\Lambda\setminus\Lambda_0$. In fact, $\mathcal{L}(U)\ne\mathcal{L}(V)$ but $\mathcal{L}(U_i,U_j)=\mathcal{L}(V_i,V_j)$ for all $i\ne j$ (because of exchangeability).
\end{ex}

We next give conditions for $H$ to be a distribution function.

\begin{thm}\label{zd4g8}
$H$ is a distribution function on $\mathbb{R}^{2p}$ whenever

\vspace{0.2cm}

\begin{itemize}

\item[(j)] $D_i$ is of class $C^2$ for each $i\in I$;

\vspace{0.2cm}

\item[(jj)] $C$ has a density $\varphi$ with respect to $m_p$;

\vspace{0.2cm}

\item[(jjj)] $\varphi$ is of class $C^p$ and, at each point $u\in[0,1]^{2p}$, one obtains
\begin{gather*}
\frac{\partial^p}{\partial u_{2p}\ldots\partial u_{p+1}}\,\,\varphi\Bigl[D_1(u_1,u_{p+1}),\ldots,D_p(u_p,u_{2p})\Bigr]\,\,\prod_{i=1}^p\frac{\partial}{\partial u_i}D_i(u_i,u_{p+i})\ge 0.
\end{gather*}
\end{itemize}

\noindent Under such conditions, one also obtains $\lambda_H\ll m_{2p}$ whenever $\mathcal{L}(X_i)\ll m_1$ for each $i\in I$.

\end{thm}

\medskip

Condition (jjj) is a technical constraint, required to guarantee the existence and positivity of the partial derivatives of $H$, and has no heuristic interpretation (known to us). We also recall that $\mathcal{L}(X_i)\ll m_1$ means that the probability distribution of $X_i$ is absolutely continuous with respect to the Lebesgue measure $m_1$.

\medskip

The proof of Theorem \ref{zd4g8} is deferred to the Appendix. Here, we give three final examples.

\medskip

\begin{ex} {\bf (Asymmetric copulas).}\label{e6n8h}
Suppose that $H$ is a distribution function on $\mathbb{R}^{2p}$. If $D_i$ is not symmetric for some $i\in I$, as we assume, then usually $\lambda_H\notin\Lambda$. However, Theorem \ref{c4b7} implies that $\lambda_H$ satisfies condition \eqref{uhb}. Therefore, Theorem \ref{n7x} yields
\begin{gather*}
\lambda:=\frac{\sum_{f\in\mathcal{F}}\,\lambda_H\circ f^{-1}}{2^p}\in\Lambda.
\end{gather*}
Furthermore, the distribution function of $\lambda$, say $G$, can be written explicitly as
\begin{gather*}
G(x)=\frac{\sum_{f\in\mathcal{F}}\,H\bigl[f(x)\bigr]}{2^p}\quad\quad\text{for all }x\in\mathbb{R}^{2p}.
\end{gather*}
Suppose now that $C,D_1,\ldots,D_p$ satisfy conditions (j)-(jj)-(jjj) and $\mathcal{L}(X_i)\ll m_1$ for each $i\in I$. Then, not only $H$ is a distribution function, but $\lambda_H\ll m_{2p}$. Hence, one can let $\pi=\lambda_H$ in Example \ref{z2w9k}.
\end{ex}

\medskip

\begin{ex} {\bf (An open problem).}\label{z4f7hq}
In principle, a knockoff $\widetilde{X}$ should be ``as independent of $X$ as possible". Thus, it is tempting to let
\begin{gather*}
D_i(u)=u_1\,u_2\quad\quad\text{for all }u\in [0,1]^2\text{ and }i\in I.
\end{gather*}
In this case, $D_1,\ldots,D_p$ are symmetric and, for all $i\in I$ and $x\in\mathbb{R}^{2p}$,
\begin{gather*}
\lim_{x_j\rightarrow\infty,\,j\in J_i}H(x)=F_i(x_i)\,F_i(x_{p+i})\quad\quad\text{where }J_i=\{1,\ldots,2p\}\setminus\{i,p+i\}.
\end{gather*}
Therefore, {\em if} $H$ is a distribution function and $(X,\widetilde{X})\sim\lambda_H$, then
\begin{gather*}
\widetilde{X}\text{ is a knockoff copy of }X\text{ and }\widetilde{X}_i\text{ is independent of }X_i\text{ for each }i\in I.
\end{gather*}
Thus, a (natural) question is: {\em If each $D_i$ is the independence copula, under what conditions $H$ is a distribution function} ? Some partial answers are available. For instance, $H$ is a distribution function if $C$ admits a smooth density $\varphi$ (with respect to $m_p$) such that
\begin{gather*}
\frac{\partial^p}{\partial u_{2p}\ldots\partial u_{p+1}}\,\,\varphi\Bigl(u_1u_{p+1},\ldots,u_pu_{2p}\Bigr)\,\,\prod_{i=1}^pu_{p+i}\ge 0.
\end{gather*}
Or else, $H$ is a distribution function if $C$ is Archimedean with a suitable generator $\psi$ (just let $\psi_i(x)=\exp(-x)$ in condition \eqref{z3ef6y} of Example \ref{ub5s0k}). To our knowledge, however, a general answer to the above question is still unknown.
\end{ex}

\medskip

\begin{ex} {\bf (Archimedean copulas).}\label{ub5s0k}
An {\em Archimedean generator} is a continuous and strictly decreasing function $\psi:[0,\infty)\rightarrow (0,1]$ such that $\psi(0)=1$ and $\lim_{x\rightarrow\infty}\psi(x)=0$. By convention, we let $\psi(\infty)=0$ and $\psi^{-1}(0)=\infty$.

Suppose $C$ is Archimedean with generator $\psi$, that is,
$$C(u)=\psi\Bigl(\,\sum_{i=1}^p\psi^{-1}(u_i)\Bigr)\quad\quad\text{for all }u\in [0,1]^p.$$
Suppose also that $\psi$ has derivatives up to order $2p$ on $(0,\infty)$ and
\begin{gather*}
(-1)^k\,\psi^{(k)}\ge 0\quad\quad\text{for }k=1,\ldots,2p,
\end{gather*}
where $\psi^{(k)}$ denotes the $k$-th derivative of $\psi$. In view of \cite[Cor. 2.1]{MN}, the latter condition implies that
$$C^*(u)=\psi\Bigl(\,\sum_{i=1}^{2p}\psi^{-1}(u_i)\Bigr),\quad\quad u\in [0,1]^{2p},$$
is a $2p$-copula. Therefore, $H$ is a distribution function on $\mathbb{R}^{2p}$ as far as $D_1,\ldots,D_p$ are Archimedean with the same generator as $C$. In this case, in fact,
\begin{gather*}
H(x)=\psi\Bigl\{\,\sum_{i=1}^p\psi^{-1}\bigl(D_i\bigl(F_i(x_i),F_i(x_{p+i})\bigr)\bigr)\Bigr\}
\\=\psi\Bigl\{\,\sum_{i=1}^p\psi^{-1}(F_i(x_i))+\sum_{i=1}^p\psi^{-1}(F_i(x_{p+i}))\Bigr\}
\\=C^*\Bigl\{F_1(x_1),\ldots,F_p(x_p),F_1(x_{p+1}),\ldots,F_p(x_{2p})\Bigr\}\quad\quad\text{for all }x\in\mathbb{R}^{2p}.
\end{gather*}
In addition, since $D_1,\ldots,D_p$ are symmetric, one also obtains $\lambda_H\in\Lambda$.

More generally, suppose that $D_i$ is Archimedean with generator $\psi_i$ for each $i\in I$. Then, $H$ is a distribution function and $\lambda_H\in\Lambda$ provided
\begin{gather}\label{z3ef6y}
(-1)^k\,\psi_i^{(k)}\ge 0\quad\text{and}\quad (-1)^{k-1}\,\bigl(\psi^{-1}\circ\psi_i\bigr)^{(k)}\ge 0
\end{gather}
for all $i\in I$ and $k=1,\ldots,2p$; see \cite[p. 190]{OOS} and \cite[p. 297]{ST}. If all the generators $\psi,\psi_1,\ldots,\psi_p$ belong to the same parametric family, such us the Gumbel or the Clayton, condition \eqref{z3ef6y} reduces to a simple restriction on the parameters; see \cite{ELM}.
\end{ex}

\vspace{0.2cm}

A last general remark is that the idea underlying Theorems \ref{c4b7} and \ref{zd4g8} could be realized, possibly in a better way, involving special types of copulas. For instance, a possibility could be using pair copulas; see e.g. \cite{ACFB09}.

\section{Conditional independence}\label{2wx8m}

To build a (reasonable) knockoff is not hard if $X$ is conditionally independent given some random element $Z$. We begin by making this claim precise.

\vspace{0.2cm}

\begin{thm}\label{m1m}
Suppose that, for some random element $Z$, one obtains
\begin{gather}\label{d7m7s}
P\bigl(X_1\in A_1,\ldots,X_p\in A_p\bigr)=E\Bigl\{\prod_{i=1}^pP(X_i\in A_i\mid Z)\Bigr\}
\end{gather}
for all $A_1,\ldots,A_p\in\mathcal{B}_1$. Let $\lambda$ be the (only) probability measure on $\mathcal{B}_{2p}$ such that
\begin{gather*}
\lambda(A_1\times\ldots\times A_{2p})=E\left\{\prod_{i=1}^pP(X_i\in A_i\mid Z)\,\prod_{i=1}^pP(X_i\in A_{p+i}\mid Z)\right\}
\end{gather*}
whenever $A_i\in\mathcal{B}_1$ for all $i=1,\ldots,2p$. Then, $\lambda\in\Lambda$.
\end{thm}

\begin{proof}
For all $A_1,\ldots,A_{2p}\in\mathcal{B}_1$, define
\begin{gather*}
\lambda_0(A_1\times\ldots\times A_{2p})=E\left\{\prod_{i=1}^pP(X_i\in A_i\mid Z)\,\prod_{i=1}^pP(X_i\in A_{p+i}\mid Z)\right\}.
\end{gather*}
Such a $\lambda_0$, defined on
\begin{gather*}
\mathcal{R}=\bigl\{A_1\times\ldots\times A_{2p}:A_i\in\mathcal{B}_1,\,i=1,\ldots,2p\bigr\},
\end{gather*}
uniquely extends to a probability measure $\lambda$ on $\mathcal{B}_{2p}$. By definition,
\begin{gather*}
\lambda\circ f^{-1}(A)=\lambda_0\circ f^{-1}(A)=\lambda_0(A)=\lambda(A)
\end{gather*}
whenever $f\in\mathcal{F}$ and $A\in\mathcal{R}$. Hence, $\lambda\in\mathcal{P}$. Finally, if $A_i=\mathbb{R}$ for $i>p$, condition \eqref{d7m7s} yields
\begin{gather*}
\lambda(A_1\times\ldots\times A_p\times\mathbb{R}^p)=E\Bigl\{\prod_{i=1}^pP(X_i\in A_i\mid Z)\Bigr\}=P\bigl(X_1\in A_1,\ldots,X_p\in A_p\bigr).
\end{gather*}
Therefore, $\lambda\in\Lambda$.
\end{proof}

In real problems, to take advantage of Theorem \ref{m1m}, one needs to select a random element $Z$ satisfying condition \eqref{d7m7s}. As an extreme example, suppose $Z=X$. Then, condition \eqref{d7m7s} holds and $P(X_i\in A_i\mid X)=1_{A_i}(X_i)$ a.s. Therefore,
\begin{gather*}
\lambda(A_1\times\ldots\times A_{2p})=E\left\{\prod_{i=1}^p1_{A_i}(X_i)\,\prod_{i=1}^p1_{A_{p+i}}(X_i)\right\}
\\=P\bigl(X_1\in A_1\cap A_{p+1},\ldots,X_p\in A_p\cap A_{2p}\bigr).
\end{gather*}
Such a $\lambda$ is precisely the probability distribution of the trivial knockoff $(X,X)$ (namely, $\widetilde{X}=X$). Thus, as it could be guessed, $Z=X$ is not a good choice. We now consider some better choices.

\begin{ex}\textbf{(Stable laws).}\label{j9i}
Let $U=(U_1,\ldots,U_p)$ and $Z=(Z_1,\ldots,Z_p)$ be $p$-variate random variables, with $U$ independent of $Z$ and $U_1,\ldots,U_p$ independent among them. Then, condition \eqref{d7m7s} holds whenever
\begin{gather*}
X=U+Z.
\end{gather*}
As an example, fix $\alpha\in (0,2]$ and suppose $U_i\sim\mathcal{S}(a_i,b_i)$ for all $i$. According to Subsection \ref{g7n9k}, this means that $U_i$ has a symmetric $\alpha$-stable distribution with parameters $a_i\in\mathbb{R}$ and $b_i>0$. For $A\in\mathcal{B}_1$, write $\mathcal{S}(a,b)(A)$ to denote the value attached to $A$ by the probability measure $\mathcal{S}(a,b)$. In this notation, since $U_i+c\sim\mathcal{S}(a_i+c,\,b_i)$ for all $c\in\mathbb{R}$, one obtains
\begin{gather*}
P(X_1\in A_1,\ldots,X_p\in A_p\mid Z)=\prod_{i=1}^p\mathcal{S}(a_i+Z_i,b_i)(A_i)\quad\quad{a.s.}
\end{gather*}
Hence, Theorem \ref{m1m} implies $\lambda\in\Lambda$ where
\begin{gather*}
\lambda(A_1\times\ldots\times A_{2p})=E\left\{\prod_{i=1}^p\mathcal{S}(a_i+Z_i,b_i)(A_i)\,\prod_{i=1}^p\mathcal{S}(a_i+Z_i,b_i)(A_{p+i})\right\}.
\end{gather*}
\end{ex}

\medskip

\begin{ex}\textbf{(Normal distributions).}\label{s4b8i1a}
As a special case of Example \ref{j9i} (with $\alpha=2$) suppose $X\sim\mathcal{N}(\mu,\Sigma)$. Let $D$ be a diagonal matrix such that $\Sigma-D$ is semidefinite positive and $d_{ii}\ge 0$ for all $i$, where $d_{ii}$ is the $i$-th diagonal element of $D$. Then, one can take $U\sim\mathcal{N}(0,D)$ and $Z\sim\mathcal{N}(\mu,\Sigma-D)$. The conditional distribution of $X$ given $Z$ is $\mathcal{N}(Z,D)$. Since $D$ is diagonal, $X_1,\ldots,X_p$ are conditionally independent, given $Z$, with $X_i\sim\mathcal{N}(Z_i,d_{ii})$. Define
\begin{gather*}
\lambda(A_1\times\ldots\times A_{2p})=E\left\{\prod_{i=1}^p\mathcal{N}(Z_i,d_{ii})(A_i)\,\prod_{i=1}^p\mathcal{N}(Z_i,d_{ii})(A_{p+i})\right\}.
\end{gather*}
Then, by Theorem \ref{m1m}, there is a knockoff copy $\widetilde{X}$ of $X$ such that $(X,\widetilde{X})\sim\lambda$. Finally, it is easily seen that
\begin{gather*}
\lambda=\mathcal{N}(\mu^*,G)\quad\text{where}\quad\mu^*=\left(
                           \begin{array}{c}
                             \mu \\
                             \mu \\
                           \end{array}
                         \right)\text{ and }G=\left(
    \begin{array}{cc}
      \Sigma & \Sigma-D \\
      \Sigma-D & \Sigma \\
    \end{array}
  \right).
\end{gather*}

A concrete example (suggested by an anonymous referee) is the so called ``equicorrelated" Gaussian distribution, namely, $\sigma_{ii}=b$ and $\sigma_{ij}=a$ for all $i$ and all $j\ne i$, where $0<a<b$ are fixed constants. In this case, it suffices to take $d_{ii}\in (0,b-a)$ for all $i$.
\end{ex}

\vspace{0.2cm}

The probability $\lambda$ obtained in Example \ref{s4b8i1a} is already known to be an element of $\Lambda$; see e.g. \cite[p. 559]{CFJL18}. Instead, in the next example, Theorem \ref{m1m} yields a new knockoff distribution.

\vspace{0.2cm}

\begin{ex}\textbf{(Mixtures of normal distributions).}
Let
\begin{gather*}
X=ZU
\end{gather*}
where $Z$ is a random $p\times p$ diagonal matrix and $U$ a $p$-dimensional column vector. Suppose $U\sim\mathcal{N}(0,I)$ and $Z$ independent of $U$. Then, the probability distribution of $X$ can be written as
\begin{gather*}
P(X\in A)=E\Bigl\{\mathcal{N}(0,ZZ)(A)\Bigr\}\quad\quad\text{for all }A\in\mathcal{B}_p.
\end{gather*}
Probability distributions of this type play a role in various frameworks. For instance, they arise as the limit laws in the CLT for exchangeable random variables; see e.g. \cite[Sect. 3]{BPR04}. In any case, since $ZZ$ is diagonal, $X_1,\ldots,X_p$ are conditionally independent given $Z$ with $X_i\sim\mathcal{N}(0,Z_{ii}^2)$. Hence, Theorem \ref{m1m} implies $\lambda\in\Lambda$ where
\begin{gather*}
\lambda(A_1\times\ldots\times A_{2p})=E\left\{\prod_{i=1}^p\mathcal{N}(0,Z_{ii}^2)(A_i)\,\prod_{i=1}^p\mathcal{N}(0,Z_{ii}^2)(A_{p+i})\right\}.
\end{gather*}
\end{ex}

\medskip

A further example, where conditional independence is exploited to obtain a knockoff, is in \cite{BSSC20}.

\medskip

\vspace{0.2cm}

In applications, to assign $\mathcal{L}(X)$ is one of the main statistician's tasks. Hence, a reasonable strategy is to model $X$ so as to realize conditional independence, with respect to some latent variable $Z$, and then to obtain a knockoff $\widetilde{X}$ via Theorem \ref{m1m}. As already noted, the advantage is twofold. On one hand, conditional independence is easy to be realized and able to describe various real situations. On the other hand, to build $\widetilde{X}$ is straightforward whenever $X$ is conditionally independent. In the rest of this section, the statistician is assumed to adopt this strategy. Thus, he/she decides to model $X$ as conditionally independent with respect to some $Z$. Note that, in this framework, $\mathcal{L}(X)$ {\em is regarded as a statistician's choice (and not as an external constraint to be satisfied)}. The next example is fundamental.

\vspace{0.2cm}

\begin{ex}\label{m2mx} {\bf (Parametric constructions of knockoffs).}
Suppose $X$ is modeled as
\begin{gather*}
P\bigl(X_1\in A_1,\ldots,X_p\in A_p\bigr)=\int_\Theta\,\prod_{i=1}^pQ_i(A_i,\theta)\,\gamma(d\theta),
\end{gather*}
where $Q_1(\cdot,\theta),\ldots,Q_p(\cdot,\theta)$ are probabilities on $\mathcal{B}_1$, indexed by some parameter $\theta\in\Theta$, and $\gamma$ is a mixing probability on $\Theta$. As an example, one could take
\begin{gather*}
Q_i(\cdot,\theta)=\mathcal{N}(\mu_i,\sigma^2_i)\quad\text{and}\quad\theta=(\mu_1,\ldots,\mu_p,\sigma^2_1,\ldots,\sigma^2_p).
\end{gather*}
In this case, $\gamma$ would be a probability measure on $\Theta=\mathbb{R}^p\times (0,\infty)^p$.

More generally, fix a $\sigma$-finite measure $\nu_i$ on $\mathcal{B}_1$ and suppose $Q_i(\cdot,\theta)$ has a density $f_i(\cdot,\theta)$ with respect to $\nu_i$, namely
\begin{gather*}
Q_i(A,\theta)=\int_Af_i(t,\theta)\,\nu_i(dt)\quad\quad\text{ for all }i\in I,\,A\in\mathcal{B}_1\text{ and }\theta\in\Theta.
\end{gather*}
Define $\lambda$ to be the probability measure on $\mathcal{B}_{2p}$ with density $q$ with respect to $\nu\times\nu$, where $\nu=\nu_1\times\ldots\times\nu_p$ and
\begin{gather*}
q(y)=q(y_1,\ldots,y_{2p})=\int_\Theta\,\prod_{i=1}^pf_i(y_i,\theta)\,\prod_{i=1}^pf_i(y_{p+i},\theta)\,\gamma(d\theta)\quad\quad\text{for all }y\in\mathbb{R}^{2p}.
\end{gather*}
Then, $\lambda\in\Lambda$ because of Theorem \ref{m1m}. Therefore, after observing $X=x$, a value $\widetilde{x}$ for the knockoff $\widetilde{X}$ can be drawn from the conditional density
\begin{gather*}
\frac{q(x,\widetilde{x})}{h(x)},
\end{gather*}
where $x,\,\widetilde{x}\in\mathbb{R}^p$ and $h(x)=\int_\Theta\,\prod_{i=1}^pf_i(x_i,\theta)\,\gamma(d\theta)$ is the marginal density of $X$.
\end{ex}

\vspace{0.2cm}

Example \ref{m2mx} is general enough to cover a wide range of real situations.

We now briefly discuss the choice of $\gamma$. It may be helpful to recall that, once $Q_1(\cdot,\theta),\ldots,Q_p(\cdot,\theta)$ have been selected, to choose $\gamma$ is equivalent to choose the probability distribution of $X$.

\vspace{0.2cm}

\begin{ex} {\bf (Choice of $\gamma$).}
It is tempting to regard the mixing measure $\gamma$ as a prior distribution. Even if not mandatory, this interpretation is helpful. Hence, in the sequel, $\gamma$ is referred to as {\em the prior}. Let $Q_i(\cdot,\theta)$, $f_i(\cdot,\theta)$ and $\lambda\in\Lambda$ be as in Example \ref{m2mx}. Two (distinct) criterions to select $\gamma$ are as follows.

Roughly speaking, $\gamma$ tunes the dependence between $X$ and $\widetilde{X}$, where $\widetilde{X}$ is such that $\mathcal{L}(X,\widetilde{X})=\lambda$. Define in fact
\begin{gather*}
Q(\cdot,\theta)=Q_1(\cdot,\theta)\times\ldots\times Q_p(\cdot,\theta).
\end{gather*}
Then, $Q(\cdot,\theta)$ is a probability measure on $\mathcal{B}_p$ and
\begin{gather}\label{w3z09}
P(X\in A,\,\widetilde{X}\in B)-P(X\in A)\,P(\widetilde{X}\in B)=
\\=\int_\Theta\,Q(A,\theta)\,Q(B,\theta)\,\gamma(d\theta)-\int_\Theta\,Q(A,\theta)\,\gamma(d\theta)\,\int_\Theta\,Q(B,\theta)\,\gamma(d\theta)\notag
\end{gather}
for all $A,\,B\in\mathcal{B}_p$. Thus, a first criterion is to choose $\gamma$ so as to make \eqref{w3z09} small for some $A$ and $B$. This is just a rough and naive indication, difficult to realize in practice, but it may be potentially useful.

To state the second criterion, denote by $h_\gamma$ the marginal density of $X$ when the prior is $\gamma$, namely
\begin{gather*}
h_\gamma(x)=\int_\Theta\,\prod_{i=1}^pf_i(x_i,\theta)\,\gamma(d\theta)\quad\quad\text{for all }x\in\mathbb{R}^p.
\end{gather*}
Suppose now that $X=x$ is observed. Then, $h_\gamma$ can be seen as the integrated likelihood of $x$ with respect to the prior $\gamma$. From a Bayesian point of view, it is desirable that $h_\gamma(x)$ is high. Therefore, a second criterion is to choose $\gamma$ so as to maximize the map $\gamma\mapsto h_\gamma(x)$. For instance, the choice between two conflicting priors $\gamma_1$ and $\gamma_2$ could be seen as a model selection problem. Accordingly, we could choose between $\gamma_1$ and $\gamma_2$ based on the {\em Bayes factor} $h_{\gamma_1}(x)/h_{\gamma_2}(x)$. A practical advantage is that we can profit on the broad literature on Bayes factors and related topics.
\end{ex}

\vspace{0.2cm}

Another useful feature of Example \ref{m2mx} is highlighted in the next example.

\vspace{0.2cm}

\begin{ex}\label{s6st8} {\bf (Uncorrelated knockoffs).}
Under some assumptions on $Q_i(\cdot,\theta)$, one obtains
\begin{gather*}
\text{cov}(X_i,\widetilde{X}_i)=0\quad\quad\text{for all }i\in I\text{ and all priors }\gamma.
\end{gather*}
Fix in fact $i\in I$ and suppose the mean of $Q_i(\cdot,\theta)$ exists and does not depend on $\theta$, say
\begin{gather*}
\int_\mathbb{R}t\,Q_i(dt,\theta)=a_i\quad\quad\text{for some }a_i\in\mathbb{R}\text{ and all }\theta\in\Theta.
\end{gather*}
Then, independently of $\gamma$, Fubini's theorem yields
\begin{gather*}
\text{cov}(X_i,\widetilde{X}_i)=\int_\Theta a_i^2\,d\gamma-\left(\int_\Theta a_i\,d\gamma\right)^2=a_i^2-a_i^2=0.
\end{gather*}
For instance, cov$(X_i,\widetilde{X}_i)=0$ provided $Q_i(\cdot,\theta)=\mathcal{N}(0,\sigma^2_i(\theta))$ for all $\theta$.
\end{ex}

\vspace{0.2cm}

We conclude our discussion of Example \ref{m2mx} with a practical example.

\vspace{0.2cm}

\begin{ex}\label{m3mxuv} {\bf (Conditionally independent Poisson data).} Let $\theta=(\theta_1,\ldots,\theta_p)$ and $Q_i(\cdot,\theta)$ a Poisson distribution with parameter $\theta_i$. We consider two different choices of the prior $\gamma$.

\vspace{0.2cm}

\noindent First, let $\gamma=\gamma_1\times\ldots\times\gamma_p$ where each $\gamma_i$ is a Gamma distribution with parameters $a_i$ and $b_i$. In this case, since $\theta_1,\ldots,\theta_p$ are independent under $\gamma$, the calculations are straightforward:

\begin{gather*}
q(y_1,\ldots,y_{2p})=\prod_{i=1}^p \int_0^{\infty}\left( \frac{\theta_i^{y_i}}{y_i!}e^{-\theta_i} \frac{\theta_i^{y_{i+p}}}{y_{i+p}!}e^{-\theta_i}\right)\,\frac{b_i^{a_i}}{\Gamma (a_i)}\theta_i^{a_i-1} e^{-b_i\theta_i}d\theta_i
\\=\prod_{i=1}^{p}\frac{1}{y_i!y_{i+p}!}\frac{b_i^{a_i}}{\Gamma (a_i)}\frac{\Gamma(a_i+y_i+y_{i+p})}{(b_i+2)^{a_i+y_i+y_{i+p}}}.
\end{gather*}

\noindent Similarly,

$$h(y_1,\ldots,y_p)=\prod_{i=1}^{p}\frac{1}{y_i!}\frac{b_i^{a_i}}{\Gamma (a_i)}\frac{\Gamma(a_i+y_i)}{(b_i+1)^{a_i+y_i}}.$$

\noindent Therefore, after observing $X=x$, a value $\widetilde{x}$ for the knockoff $\widetilde{X}$ can be drawn from the conditional density

$$\frac{q(x,\tilde{x})}{h(x)}=\prod_{i=1}^{p}\frac{1}{\tilde{x}_i!}\frac{\Gamma(a_i+x_i+\tilde{x}_i)}{\Gamma(a_i+x_i)}\frac{(b_i+1)^{a_i+x_i}}{(b_i+2)^{a_i+x_i+\tilde{x}_i}}.$$

\noindent Second, let $\gamma$ be a Dirichlet distribution with parameters $a_1,\ldots,a_p$. Denote by
$$S=\left\{\theta\in\mathbb{R}^p:\theta_i\ge 0\text{ for all }i\text{ and }\sum_{i=1}^p\theta_i=1\right\}$$
the $p$-dimensional simplex, and by
$$m(n_1,\ldots,n_p)=\int_S\,\theta_1^{n_1}\ldots\theta_p^{n_p}\,\gamma(d\theta)$$
the mixed moment of $\gamma$ of order $(n_1,\ldots,n_p)$. Explicit formulae for $m(n_1,\ldots,n_p)$ are available; see e.g. \cite{KBJ2000}, page 488, equation (49.7). Since $\gamma(S)=1$, one obtains

\begin{gather*}
q(y_1,\ldots,y_{2p})=\int_S\,\exp\left(-2\sum_{i=1}^p\theta_i\right)\,\prod_{i=1}^p\theta_i^{y_i+y_{i+p}}\,\prod_{i=1}^p\frac{1}{y_i!y_{i+p}!}\,\gamma(d\theta)\\
=e^{-2}\,m(y_1+y_{p+1},\ldots,y_p+y_{2p})\,\prod_{i=1}^p\frac{1}{y_i!y_{i+p}!}.
\end{gather*}

\noindent Similarly,

$$h(y_1,\ldots,y_p)=e^{-1}\,m(y_1,\ldots,y_p)\,\prod_{i=1}^p\frac{1}{y_i!}.$$

\noindent Hence, the conditional density of $\widetilde{X}$ given $X=x$ can be written as

$$\frac{q(x,\tilde{x})}{h(x)}=e^{-1}\,\,\frac{m(x_1+\tilde{x}_1,\ldots,x_p+\tilde{x}_p)}{m(x_1,\ldots,x_p)}\,\prod_{i=1}^{p}\frac{1}{\tilde{x}_i!}.$$

\end{ex}

\section{Sampling strategies}\label{e5ttf67n}


In Sections \ref{c5g8n} and \ref{2wx8m}, exploiting copulas and conditional independence, two general methods for constructing knockoffs have been introduced. In this section, having applications in mind, such methods are translated into practical algorithms. Two classical MCMC algorithms, the Metropolis-Hastings sampler and the
Gibbs sampler via data augmentation, are proposed. Obviously, our proposals are not the only possible ones. The literature on MCMC is huge (see e.g. \cite{BGJM11}) and some better sampling strategies could be available. The only goal of this section is to point out that the material of Sections \ref{c5g8n}-\ref{2wx8m} can be easily used in applied settings.

\medskip

We denote by $x=(x_1,\ldots,x_p)$ and $\widetilde{x}=(\widetilde{x}_1,\ldots,\widetilde{x}_p)$ two points of $\mathbb{R}^p$. Here, $x$ should be regarded as the observed value of $X$ and $\widetilde{x}$ as the value to be sampled of the knockoff $\widetilde{X}$.

\subsection{A Metropolis-Hastings approach to copula knockoffs}

In the notation of Section \ref{c5g8n}, we assume that $C,D_1,\ldots,D_p,F_1,\ldots,F_p$ are all absolutely continuous with respect to the Lebesgue measure of appropriate dimension. Algorithm 1 provides a strategy to sample $\widetilde{x}$ via the copula construction of Section \ref{c5g8n}.

\begin{algorithm}
\caption{Copula knockoffs: general algorithm}\label{alg:cap1}
\begin{algorithmic}
\State 1.  Choose the distribution functions $F_1,\ldots,F_p$ on $\mathbb{R}$, a $p$-copula C and a family of 2-copulas $D_1,\dots,D_p$ in such a way that
\begin{gather*}
H(x,\widetilde{x})=C\Bigl[D_1\bigl(F_1(x_1),F_1(\widetilde{x}_1)\bigr),\ldots,D_p\bigl(F_p(x_p),F_p(\widetilde{x}_p)\bigr)\Bigr]
\end{gather*}
is a distribution function on $\mathbb{R}^{2p}$
\State 2.  Sample $\widetilde{x}$ from the conditional density
\begin{gather*}
p(\widetilde{x}\mid x)=\frac{1}{\varphi\bigl[F_1(x_1),\ldots,F_p(x_p)\bigr]\,\prod_{i=1}^pf_i(x_i)}\,\cdot\,\frac{\partial^{2p} H}{\partial x_p\ldots\partial x_1\partial \widetilde{x}_p\ldots\partial \widetilde{x}_1}(x,\widetilde{x})
\end{gather*}
\end{algorithmic}
where $\varphi$ and $f_i$ are the densities of $C$ and $F_i$, respectively
\end{algorithm}

Sampling from $p(\widetilde{x}\mid x)$ may be not straightforward. However, since
$$p(\widetilde{x}\mid x)\propto \frac{\partial^{2p} H}{\partial x_p\ldots\partial x_1\partial \widetilde{x}_p\ldots\partial \widetilde{x}_1}(x,\widetilde{x}),$$
a Metropolis-Hastings sampler is available. One such sampler is provided by Algorithm \ref{alg:cap2}.
\begin{algorithm}
\caption{Metropolis-Hastings sampler}\label{alg:cap2}
\begin{algorithmic}
\State 1. Choose the initial value $\widetilde{x}^{(0)}$

\State 2. Choose the proposal distribution $K(\cdot|x)$ (usually a Markov Kernel)

\For{$j\gets 1$ to $M$}

\State 3. Sample $y$ from the proposal $K(\cdot|\widetilde{x}^{(j-1)})$

\State 4. Compute the acceptance probability

$$\alpha(\widetilde{x}^{(j-1)},y)=\min\left\lbrace 1,\frac{p(y\mid x)}{p(\widetilde{x}^{(j-1)}\mid x)}\frac{K(\widetilde{x}^{(j-1)}|y)}{K(y|\widetilde{x}^{(j-1)})}\right\rbrace=\min\left\lbrace 1,\frac{\frac{\partial^{2p} H}{\partial x_p\ldots\partial x_1\partial \widetilde{x}_p\ldots\partial \widetilde{x}_1}(x,y)}{\frac{\partial^{2p} H}{\partial x_p\ldots\partial x_1\partial \widetilde{x}_p\ldots\partial \widetilde{x}_1}(x,\widetilde{x}^{(j-1)})}\frac{K(\widetilde{x}^{(j-1)}|y)}{K(y|\widetilde{x}^{(j-1)})}\right\rbrace$$

\State 5. Set $\widetilde{x}^{(j)}=y$ with probability $\alpha(\widetilde{x}^{(j-1)},y)$  and $\widetilde{x}^{(j)}=\widetilde{x}^{(j-1)}$ with probability $1-\alpha(\widetilde{x}^{(j-1)},y)$

 \EndFor

       \State 6. Return the sample $\widetilde{x}^{(j)}$,  $j=1,\dots,M$

\end{algorithmic}
\end{algorithm}
\subsection{A data augmentation approach to conditional independence knockoffs}

We now outline how to sample knockoffs via the conditional independence  strategy proposed in Example \ref{m2mx}.  The main steps of the procedure are summarized by Algorithm \ref{alg:cap3}.

\medskip

\begin{algorithm}
\caption{Conditional independence knockoffs: general algorithm}\label{alg:cap3}
\begin{algorithmic}
\State 1.  Choose a density $f_i(\cdot,\theta)$, with respect to some reference measure $\nu_i$, for $X_i$,  $i=1,\dots,p$

\State 2.  Choose a mixing probability $\gamma(d\theta)$

\State 3. Compute
$$q(x,\widetilde{x})=\int_\Theta\,\prod_{i=1}^pf_i(x_i,\theta)\,\prod_{i=1}^pf_i(\widetilde{x}_i,\theta)\,\gamma(d\theta)\quad\text{ and }\quad h(x)=\int_\Theta\,\prod_{i=1}^pf_i(x_i,\theta)\,\gamma(d\theta)$$

\State 4.  Sample $\widetilde{x}$ from the conditional density $q(x,\widetilde{x})/h(x)$
\end{algorithmic}
\end{algorithm}

Step 4 of Algorithm \ref{alg:cap3} could be difficult since the numerator and denominator of the conditional density $q(x,\widetilde{x})/h(x)$ are often not in closed form. Hence, $q(x,\widetilde{x})/h(x)$ may be not explicit and computational methods come to the fore. The Metropolis-Hastings could be problematic since we have to evaluate integrals in the acceptance rate. This can be time consuming. An alternative approach is a Data Augmentation strategy where both $\widetilde{x}$ and $\theta$ are sampled at the same time.

Suppose $\Theta$ is an open subset of $\mathbb{R}^k$ for some $k$, and $\gamma$ has a density with respect to Lebesgue measure on $\Theta$, say $\gamma(d\theta)=p(\theta)\,d\theta$. Then,
$$\frac{q(x,\widetilde{x})}{h(x)}=\int_\Theta \frac{\bar{q}(x,\widetilde{x},\theta)}{h(x)}\,d\theta\quad\text{where}\quad\bar{q}(x,\widetilde{x},\theta)=p(\theta)\,\prod_{i=1}^pf_i(x_i,\theta)\,\prod_{i=1}^pf_i(\widetilde{x}_i,\theta).$$
This is quite convenient since it makes easier to implement a Gibbs sampler on the augmented space with $\theta$.  The full conditional distributions are straightforward by noting that
$$\frac{\bar{q}(x,\widetilde{x},\theta)}{h(x)}\propto p(\theta)\,\prod_{i=1}^pf_i(x_i,\theta)\,\prod_{i=1}^pf_i(\widetilde{x}_i,\theta).$$
Algorithm \ref{alg:cap4} provides a Gibbs sampler for $\bar{q}(x,\widetilde{x},\theta)$.

\begin{algorithm}
\caption{Data Augmentation sampler}\label{alg:cap4}
\begin{algorithmic}
\State 1. Choose the initial value $\theta^{(0)}$

\For{$j\gets 1$ to $M$}

\For{$i\gets 1$ to $p$}

        \State 2. Sample $\widetilde{x}^{(j)}_i|\,\theta^{(j-1)}\sim f_i(\cdot,\theta^{(j-1)})$

       \EndFor

\State 3. Sample $\theta^{(j)}$ from the posterior of $\theta$ given $x,\widetilde{x}^{(j)}$

 \EndFor

       \State 4. Return the sample $(\widetilde{x}^{(j)},\theta^{(j)})$,  $j=1,\dots,M$

\end{algorithmic}
\end{algorithm}

It should be noted that Algorithms \ref{alg:cap2} and \ref{alg:cap4} provide a sample $\widetilde{x}^{(1)},\ldots,\widetilde{x}^{(M)}$ of knockoffs rather than a single realization.  This could be helpful when taking into account the uncertainty intrinsic in the simulation procedure. Note also that, at each step $j$, Algorithm \ref{alg:cap4} requires to sample from the posterior of $\theta$ given $x,\widetilde{x}^{(j)}$. In some cases, this could not be an easy step. However, it is straightforward in several scenarios, such as conjugate models.

\medskip

\newpage

\begin{center}{\bf Appendix}\end{center}

\medskip

\begin{proof}[Proof of Theorem \ref{r5y7bq}]
First note that, since $\mathcal{F}$ is a group under composition,
\begin{gather*}
\sum_{f\in\mathcal{F}}\,g\circ f^{-1}=\sum_{f\in\mathcal{F}}\,g\circ f\quad\quad\text{for any real function }g\text{ on }\mathbb{R}^{2p}.
\end{gather*}

{\bf ``(a) $\Rightarrow$ (b)".} Since $\mathcal{F}$ contains the identity map, condition (a) implies $\pi\le 2^p\lambda$. Hence, $\pi$ has a density $q$ with respect to $\lambda$. Since $\lambda\in\mathcal{P}$, condition (a) also implies
\begin{gather*}
\int_A2^p\,d\lambda=2^p\,\lambda(A)=\sum_{f\in\mathcal{F}}\,\pi\circ f^{-1}(A)=\sum_{f\in\mathcal{F}}\,\int_{f^{-1}(A)}q\,d\lambda
\\=\sum_{f\in\mathcal{F}}\,\int_Aq\circ f^{-1}\,d\lambda
=\int_A\,\Bigl(\sum_{f\in\mathcal{F}}\,q\circ f\Bigr)\,d\lambda\quad\quad\text{for each }A\in\mathcal{B}_{2p}.
\end{gather*}

{\bf ``(b) $\Rightarrow$ (c)".} If $A\in\mathcal{G}$, then $A=f^{-1}(A)$ for all $f\in\mathcal{F}$, so that
\begin{gather*}
\pi(A)=\pi\bigl(f^{-1}(A)\bigr)=\int_{f^{-1}(A)}q\,d\lambda=\int_Aq\circ f^{-1}\,d\lambda\quad\quad\text{for all }f\in\mathcal{F}.
\end{gather*}
Hence, condition (b) implies
\begin{gather*}
2^p\,\pi(A)=\sum_{f\in\mathcal{F}}\,\int_Aq\circ f^{-1}\,d\lambda=\int_A\,\Bigl(\sum_{f\in\mathcal{F}}\,q\circ f\Bigr)\,d\lambda=2^p\,\lambda(A).
\end{gather*}

{\bf ``(c) $\Rightarrow$ (a)".} For each $x\in\mathbb{R}^{2p}$, define
\begin{gather*}
\mu_x=\frac{\sum_{f\in\mathcal{F}}\,\delta_{f(x)}}{2^p}
\end{gather*}
where $\delta_{f(x)}$ denotes the unit mass at the point $f(x)$. Then,
\begin{gather*}
\lambda(A)=\frac{\sum_{f\in\mathcal{F}}\,\lambda\circ f^{-1}(A)}{2^p}=\int\mu_x(A)\,\lambda(dx)
\\=\int\mu_x(A)\,\pi(dx)=\frac{\sum_{f\in\mathcal{F}}\,\pi\circ f^{-1}(A)}{2^p}\quad\quad\text{for all }A\in\mathcal{B}_{2p}
\end{gather*}
where the first equality follows from $\lambda\in\mathcal{P}$ and the third is because $\pi=\lambda$ on $\mathcal{G}$ and the map $x\mapsto\mu_x(A)$ is $\mathcal{G}$-measurable.
\end{proof}

\medskip

\begin{proof} [Proof of Theorem \ref{zd4g8}] We first recall a known fact. Let $\Phi:[0,1]^n\rightarrow [0,1]$ be a function such that $\Phi(u)=0$, if $u_i=0$ for some $i$, and $\Phi(u)=u_i$ if $u_j=1$ for all $j\ne i$. Then, $\Phi$ is an $n$-copula and $\lambda_\Phi\ll m_n$ provided $\frac{\partial^n\Phi}{\partial u_n\ldots\partial u_1}\ge 0$ on $[0,1]^n$.

After noting this fact, define
\begin{gather*}
C^*(u)=C\Bigl[D_1(u_1,u_{p+1}),\ldots,D_p(u_p,u_{2p})\Bigr]
\\=\int_0^{D_1(u_1,u_{p+1})}\ldots\int_0^{D_p(u_p,u_{2p})}\varphi(t_1,\ldots,t_p)\,dt_1\ldots dt_p\quad\quad\text{for each }u\in [0,1]^{2p}.
\end{gather*}

Let $n=2p$ and $\Phi=C^*$. By the result mentioned above, $C^*$ is a $2p$-copula and $\lambda_{C^*}\ll m_{2p}$ provided
\begin{gather}\label{x4r66yhb}
\frac{\partial^{2p}C^*}{\partial u_{2p}\ldots\partial u_1}\ge 0\quad\quad\text{everywhere on }[0,1]^{2p}.
\end{gather}
In this case, since $C^*$ is a copula, $H$ is a distribution function. Since $\lambda_{C^*}\ll m_{2p}$, one also obtains $\lambda_H\ll m_{2p}$ whenever $\mathcal{L}(X_i)\ll m_1$ for each $i\in I$. Therefore, it suffices to prove condition \eqref{x4r66yhb}. In turn, \eqref{x4r66yhb} follows from condition (jjj) after noting that
\begin{gather*}
\frac{\partial^{2p}C^*}{\partial u_{2p}\ldots\partial u_1}=\frac{\partial^p}{\partial u_{2p}\ldots\partial u_{p+1}}\,\,\frac{\partial^p C^*}{\partial u_p\ldots\partial u_1}
\\=\frac{\partial^p}{\partial u_{2p}\ldots\partial u_{p+1}}\,\,\varphi\Bigl[D_1(u_1,u_{p+1}),\ldots,D_p(u_p,u_{2p})\Bigr]\,\,\prod_{i=1}^p\frac{\partial}{\partial u_i}D_i(u_i,u_{p+i}).
\end{gather*}
\end{proof}

\medskip

\textbf{Acknowledgments:} This paper has been improved by the useful remarks of the AE and an anonymous referee.

\end{document}